\documentclass[11 pt]{article} 
\usepackage[applemac]{inputenc} 
\usepackage[pdftex]{graphicx}

\usepackage{hyperref} 

\usepackage{amsmath}
\usepackage{amsthm}
\usepackage{latexsym}
\usepackage{amssymb} 

 \def\sk{\medskip}
\def\rmd{{\mathrm{d}}}
\def\C{\mathbb{C}}

\def\Z{\mathbb{Z}}

\def\calE{{\mathcal{E}}}
\def\bfc{{\mathbf{c}}}

\def\K{{\mathbb K}}
\def\calO{{\mathcal O}}
\def\Card{{\mathrm{Card}}}

\def\og{\leavevmode\raise.3ex\hbox{$\scriptscriptstyle
\langle\!\langle$}}

\def\fg{\leavevmode\raise.3ex\hbox{$\scriptscriptstyle
\,\rangle\!\rangle$}}

\def\atop#1#2{
\genfrac{}{}{0pt} {} 
{#1} 
{#2}}

\newtheorem{theorem}{\indent Theorem}
\newtheorem{proposition}{\indent Proposition}
\newtheorem{lemma}{\indent Lemma}

\begin{document}

 \noindent
 ICPAM2 - GOROKA 2014

\begin{center}

 International Conference on Pure and Applied Mathematics.
 
 ``Contemporary Developments in the Mathematical Sciences as Tools for
 
 Scientific and Technological Transformation of Papua New Guinea''

\bigskip

\Large
\bf
Linear recurrence sequences 

and twisted binary forms 

\rm
 
 \bigskip

{\it Claude Levesque} and {\it Michel Waldschmidt }

\end{center}

\bigskip

\section*{Abstract}
Let 
$
 \prod_{i=1}^d (X-\alpha_i Y) \in\C[X,Y]
$
be a binary form 
and let $\epsilon_1,\dots,\epsilon_d$ be nonzero complex numbers. We consider the family of binary forms
 $
 \prod_{i=1}^d (X-\alpha_i \epsilon_i^aY)
$, $a\in\Z$,
which we write as  
\par 
 \mbox{}$\quad 
 X^d-U_1(a)X^{d-1}Y+\cdots+(-1)^{d-1} U_{d-1}(a) XY^{d-1}+(-1)^d U_d(a) Y^d.
$\\
 In this paper we study these sequences $\bigl(U_h(a)\bigr)_{a\in\Z}$ which turn out to be linear recurrence sequences.

\section*{R\'esum\'e}
Soit $
 \prod_{i=1}^d (X-\alpha_i Y) 
$
 une forme binaire de $\C[X,Y]$
et soit $\epsilon_1,\dots,\epsilon_d$ des nombres complexes non nuls. Nous consid\'erons la famille des formes binaires
 \mbox{$
 \prod_{i=1}^d (X-\alpha_i \epsilon_i^aY)
$}, $a\in\Z$,
que nous \'ecrivons sous la forme\par
 \mbox{}$\quad 
 X^d-U_1(a)X^{d-1}Y+\cdots+(-1)^{d-1} U_{d-1}(a) XY^{d-1}+(-1)^d U_d(a) Y^d.
$\\
Le but de cet article est d'\'etudier ces suites $\bigl(U_h(a)\bigr)_{a\in\Z}$ qui s'av\`erent \^etre des suites r\'ecurrentes lin\'eaires. 

\bigskip

\noindent
{\tt Keywords:} Linear recurrence sequences; binary forms; units of algebraic number fields; families of Diophantine equations; exponential polynomials

\medskip
\noindent
{\small
{\tt AMS Mathematics Subject Classification(2010):} primary 
11B37 
\\
 Secundary: 05A15 11D61 33B10 39A10 65Q30}

\normalsize
 
\section{Introduction}\label{S:Introduction}
Let us consider a binary form $F_0(X,Y)\in\C[X,Y]$ which satisfies $F_0(1,0)=1$. We write it as 
$$
F_0(X,Y)=X^d+a_1X^{d-1}Y+\cdots+a_dY^d \;=\; \prod_{i=1}^d (X-\alpha_i Y).
$$
Let $\epsilon_1,\dots,\epsilon_d$ be $d$ nonzero complex numbers not necessarily distinct. Twisting $F_0$ by the powers $\epsilon_1^a,\dots,\epsilon_d^a$ ($a\in\Z$), we obtain the family of binary forms
\begin{equation}\label{Equation:FaSomme}
F_a(X,Y)=\prod_{i=1}^d (X-\alpha_i \epsilon_i^aY),
\end{equation}
which we write as
\begin{equation}\label{Equation:FaProduit}
F_a(X,Y)=X^d-U_1(a)X^{d-1}Y+\cdots+(-1)^{d-1} U_{d-1}(a) XY^{d-1}+(-1)^d U_d(a) Y^d.
\end{equation}
Therefore
$$
U_h(0)=(-1)^h a_h \qquad (1\le h\le d).
$$

In \cite{LW-AA} and \cite{LW-Ram}, we consider some families of diophantine equations
$$
F_a(x,y)=m
$$ 
obtained in the same way from a given irreducible form $F(X,Y)$ with coefficients in $\Z$, when $\epsilon_1,\dots,\epsilon_d$ are algebraic units and when the algebraic numbers $\alpha_1 \epsilon_1,\dots,\alpha_d \epsilon_d$ are Galois conjugates with $d\ge 3$.  
The results in \cite{LW-Ram} are effective, the results in \cite{LW-AA} are more general but not effective. The next result follows from Theorem 3.3 of \cite{LW-AA}. 

\begin{theorem}
Let $K$ be a number field of degree $d\ge 3$, $S$ a finite set of places of $K$ containing the places at infinity. Denote by $\calO_S$ the ring of $S$--integers of $K$ and by 
$\calO_S^\times$ the group of $S$--units of $K$. 
Assume $\alpha_1, \ldots , \alpha_d, \epsilon_1, \ldots , \epsilon_d$
belong to $K^\times$. 
Then there are only finitely many $(x,y,a)$ in $\calO_S\times\calO_S\times\Z$ satisfying 
$$
F_a(x,y)\in\calO_S^\times,
\quad xy\not=0
\quad\hbox{and}\quad
\Card\{\alpha_1\epsilon_1^a,\dots,\alpha_d\epsilon_d^a\}\ge 3.
$$

\end{theorem}

Section $\ref{S:SuitesRecurrentesLineaires}$ is an introduction to linear recurrence sequences. 
In Section $\ref{S:FamillesFormesBinaires}$ we observe that in the general case each of the sequences $\bigl(U_h(a)\bigr)_{a\in\Z}$ coming from the coefficients of the relation $(\ref{Equation:FaProduit})$ is a linear recurrence sequence.

\section{Linear recurrence sequences }\label{S:SuitesRecurrentesLineaires}

Let us recall some well known facts about linear recurrence sequences; (see for instance \cite{WMS}, Chapter C of \cite{ST}, and also \cite{CMP}, \cite{EPSW}, \cite{Lev1}, \cite{Lev2}, \cite{vdP}). Then we apply these results to the families of binary forms given in $(\ref{Equation:FaSomme})$ and $(\ref{Equation:FaProduit})$.

\subsection{Generalities}\label{SS:SuitesRecurrentesLineaires}
Let $\K$ be a field of characteristic 0. The sequences $\bigl( u(a)\bigr)_{a\in\Z}$, with values in $\K$ and indexed by $\Z$, form a vector space
$\K^\Z$ over $\K$. Let $\bfc=(c_1,\dots,c_d)\in \K^d$ with $c_d\not=0$. The sequences, satisfying the
linear recurrence relation of order $d$ given by 
\begin{equation}\label{Equation:relationrecurrencehomogene}
u(a+d)=c_1u(a+d-1)+\cdots + c_du(a),
\end{equation}
form a $\K$--vector subspace $E_{\bfc}$ of $\K^\Z$ of dimension $d$, a natural canonical basis being given by the $d$ sequences $u_0,\dots,u_{d-1}$ defined by the initial conditions
$$
u_j(a)=\delta_{ja}\quad (0\le j,a\le d-1),
$$
$\delta_{ja}$ being the Kronecker symbol 
$$
\delta_{ja}=\begin{cases}
1&\hbox{ if $\;j=a$,}
\\
0&\hbox{ if $\;j\neq a$.}
\end{cases}
$$
For $u\in E_{\bfc}$, 
we have
$$
u=u(0)u_0+u(1)u_1+\cdots+u(d-1)u_{d-1}.
$$
By definition, the characteristic polynomial of the linear recurrence relation $(\ref{Equation:relationrecurrencehomogene})$ is
$$
P(T)=T^d-c_1T^{d-1}-\cdots-c_{d-1}T -c_d\in \K[T],
$$
where $P(0)=-c_d\neq 0$.
\sk

A sequence $u\in \K^\Z$ satisfies a linear recurrence relation of order $\le d$ if and only if the sequences 
$$
\bigl(u(a+j)\bigr)_{a\in\Z} \quad
(j=0,1,2,\dots)
$$
generate a vector space over $\K$ of dimension $\le d$. Remark that a linear recurrence relation of order $d$ may be viewed as a linear recurrence relation of order $d+s$ for any $s\ge 1$. The dimension $d_0$ of this vector space is the minimal order of the linear recurrence relation satisfied by $u$. The linear recurrence relation of order $d_0$ satisfied by $u$ is unique; the characteristic polynomial of this relation generates an ideal of $\K[T]$ and the characteristic polynomials of these linear recurrence relations satisfied by $u$ are the monic polynomials of this ideal. 
\subsection{ Decomposed characteristic polynomial }\label{SS:PolynomeCaracteristiqueDecompose}

As a preliminary step, let us assume that the polynomial $P(T)$ of degree $d$ splits completely in $\K[T]$ as a product of linear factors: 
$$
P(T)= \prod_{j=1}^\ell (T-\gamma_j)^{t_j}
$$
with $t_j\ge 1$, $\; t_1+\cdots+t_\ell=d\,$ and with nonvanishing pairwise distinct elements $\,\gamma_1,\dots,\gamma_\ell$. Let us prove that a basis of $E_{\bfc}$ is given by the $d$ sequences 
$$
\bigl( a^i \gamma_j^a\bigr)_{a\in\Z} \quad (1\le j\le \ell, \;\; 0\le i\le t_j-1).
$$
 Firstly, we will show  
that these $d$ sequences belong to the vector space $E_{\bfc}$ %
(this part was omitted in \cite{Lev2}). 
Next, we will prove that they form a linearly independent subset of $E_{\bfc}$. 

By hypothesis, for $\,1\le j\le \ell\,$ and $\,0\le i\le t_j-1$, \, the derivative of order $i$ of the polynomial $P(T)$ is vanishing at the point $\gamma_j$. Let us recall that the characteristic of $\K$ is 0. Instead of using the operator $\rmd/\rmd T$, we will use the operator $T\rmd/\rmd T$ which has the property 
$$
\left(
T\frac{\rmd}{\rmd T}
\right)^i
T^h
=h^i T^h
$$
for $i\ge 0\,$ and $\,h\ge 0$; we stipulate that $h^i=1\,$ for $\,i=h=0$. 
For  $a\in\Z$, $\,1\le j\le \ell\,$ and $\, 0\le i\le t_j-1,\,$ the equation
$$
\left(
T\frac{\rmd}{\rmd T}
\right)^i (T^aP)(\gamma_j) =0 
$$
can be written as 
$$
(a+d)^i \gamma_j^{a+d} =
\sum_{k=1}^d (a+d-k)^i c_k \gamma_j^{a+d-k} \qquad (a\in\Z),
$$
with the convention that for $k=a+d$, the term $(a+d-k)^i $ takes the value $1$ for $i=0$ and the value $0$ for $i\ge 1$.
Therefore the sequence $\bigl( a^i \gamma_j^a\bigr)_{a\in\Z}$ belongs to the vector space $E_{\bfc}$ for $\,1\le j\le \ell\,$ and $\,0\le i\le t_j-1$.

{\bf Remark.} In the literature, there are at least two further classical proofs of this fact. One is to write the linear recurrence relation in a matrix form 
$$
U(a+1)=CU(a)
$$
with 
$$
U(a)=\left(
\begin{matrix}
u(a)
\\
u(a+1)
\\
\vdots
\\
u(a+d-1)
\end{matrix}
\right),
\quad
C=
\left(
\begin{matrix}
0&1&0&\cdots&0
\\
0&0&1&\cdots&0
\\
\vdots&\vdots&\vdots&\ddots&\vdots
\\
0&0&0&\cdots&1
\\
c_d&c_{d-1}&c_{d-2}&\cdots&c_1
\end{matrix}
\right).
$$
The determinant of $I_dT-C$ (the characteristic polynomial of $C$) is nothing but $P(T)$. To obtain the result, one writes the matrix $C$ in its Jordan normal form. \sk

The other method  
consists in introducing the formal power series
$$
U(z)=\sum_{a\ge 0} u(a) z^a.
$$
One has 
$$
\left(1-\sum_{i=1}^d c_iz^i \right) U(z)=
\sum_{j=0}^{d-1} \left(
u(j)-\sum_{i=1}^j c_i u(j-i)
\right) z^j.
$$
Hence $U(z)$ is a rational fraction, with denominator 
$$
1-\sum_{i=1}^d c_iz^i=z^dP(1/z)=\prod_{j=1}^\ell (1-\gamma_jz)^{t_j},
$$
while the numerator is of degree $<d$.
This rational fraction can be rewritten using a partial fraction decomposition: 
$$
U(z)=\sum_{j=1}^\ell 
\sum_{i=0}^{t_j-1} \frac{q_{ij}}{(1-\gamma_jz)^{i+1}}\cdotp
$$
For $1\le j\le \ell,\,$ one develops $(1-\gamma_jz)^{-i-1}$ as a power series expansion to get 
$$
\frac{1}{(1-\gamma_jz)^{i+1}}=
\frac{1}{i!\gamma_j^i} \left(\frac{\rmd}{\rmd z}\right)^i \frac{1}{1-\gamma_jz}
=\sum_{a\ge 0} \frac{(a+1)(a+2)\cdots(a+i)}{i!} \gamma_j^a z^a.
$$
This allows to write $u(a)$ as a linear combination of the elements 
 $\gamma_j^a$ with coefficients being polynomials of degree $<t_j$ evaluated at $a$. \sk

Proving the linear independence of the set of the $d$ sequences 
$$
\bigl( a^i \gamma_j^a\bigr)_{a\in\Z}, \quad \hbox{ with $\,1\le j\le \ell\,$ and $\,0\le i\le t_j-1$},
$$
boils down to showing that the determinant
of the matrix

\begin{equation}\label{Equation:matriceA}
 A=
 \left(
 \begin{matrix}
 1 & \gamma_1 &\gamma_1^2&\ldots&\gamma_1^k&\ldots&\gamma_1^{d-1}
 \\[2mm]
 0 & 1 & 2\gamma_1 & \ldots&k\gamma_1^{k-1}&\ldots&(d-1)\gamma_1^{d-2}
 \\[-1mm]
 \vdots&\vdots&\vdots&\ddots&\vdots&\ddots&\vdots
 \\
 0 & 0 & 0& \ldots&
 \binom{k}{t_1-1}\gamma_1^{k-t_1+1}&\ldots& \binom{d-1}{t_1-1}\gamma_1^{d-t_1}
 \\[2mm]\hline 
 \\[-3mm]
 1 & \gamma_2 &\gamma_2^2&\ldots&\gamma_2^k&\ldots&\gamma_2^{d-1}
 \\[2mm]
 0 & 1 & 2\gamma_2 & \ldots&k\gamma_2^{k-1}&\ldots&(d-1)\gamma_2^{d-2}
 \\[-1mm]
 \vdots&\vdots&\vdots&\ddots&\vdots&\ddots&\vdots
 \\[-1mm]
 0 & 0 & 0& \ldots&
 \binom{k}{t_2-1}\gamma_2^{k-t_2+1}&\ldots& \binom{d-1}{t_2-1}\gamma_2^{d-t_2}
 \\[2mm]\hline 
 \\[-4mm]
 \vdots&\vdots&\vdots&\vdots&\vdots&\vdots&\vdots
 \\[2mm]\hline 
 \\[-4mm]
 1 & \gamma_\ell & \gamma_\ell^2&\ldots&\gamma_\ell^k&\ldots& \gamma_\ell^{d-1}
 \\[2mm]
 0 & 1 & 2\gamma_\ell & \ldots&k\gamma_\ell^{k-1}&\ldots&(d-1)\gamma_\ell^{d-2}\\ 
 \vdots&\vdots&\vdots&\ddots&\vdots&\ddots&\vdots
 \\
 0 & 0 & 0& \ldots& \binom{k}{t_\ell-1}\gamma_\ell^{k-t_\ell+1}&\ldots& \binom{d-1}{t_\ell-1}\gamma_\ell^{d-t_\ell}\\
 \end{matrix}
 \right)
\end{equation}

\noindent is different from 0. Note that $\displaystyle \binom{r}{k}=0$ for $ r<k$. 
 Let us define $s_j$ to be 
$$s_j=t_1+\cdots+t_{j-1} \quad \mbox{ for } \quad 1\le j\le \ell\,\mbox{ with } \; s_1=0. 
$$
 For $1\leq j\leq \ell$, $0\le i\le t_j-1,\;\; 0\le k\le d-1$,
 the $(s_j+i,k)$ entry of the matrix $A$ is
$$
\frac{1}{i!} \left. \left( \frac{\rmd}{\rmd T}\right)^i T^k \right|_{T=\gamma_j}
 =
 \binom{k}{i}\gamma_j^{k-i}.
 $$

As a matter of fact, $A$ is best described as being made of $\ell$ vertical blocks $A_1, A_2, \dots , A_\ell$ where for $1\leq j\leq \ell$, $A_j$ is the $t_j\times d$ matrix
 \begin{equation} \label{blocks}
 A_j=
 \left(
 \begin{matrix}
 1 & \gamma_j & \gamma_j^2&\cdots&\gamma_j^{t_j-1} & \gamma_j^{t_j}&
\cdots&\gamma_j^{d-1}
 \\[3mm]
 0 & 1 &\binom{2}{1} \gamma_j & \ldots& \binom{t_j-1}{1}\gamma_j^{t_j-2}
& \binom{t_j}{1}\gamma_j^{t_j-1}&\ldots&\binom{d-1}{1}\gamma_j^{d-2}
 \\[3mm]
0&0&1&\ldots& \binom{t_j-1}{2}\gamma_j^{t_j-3}&\binom{t_j}{2}\gamma_j^{t_j-2}&\ldots&\binom{d-1}{2}\gamma_j^{d-3}\\ 
 \vdots&\vdots&\vdots&\ddots&\vdots&\vdots&\ddots&\vdots
 \\
 0 & 0 & 0&\dots&1&\binom{t_j}{t_j-1}\gamma_j
 &\cdots& \binom{d-1}{t_j-1}\gamma_j^{d-t_j}
 \end{matrix}
 \right).
 \end{equation}

 Denote by $C_0,\dots,C_{d-1}$ the $d$ columns of $A$. Let $b_0,\dots,b_{d-1}$ be 
complex numbers such that 
 $$
 b_0C_0+\cdots+b_{d-1}C_{d-1}={\bf 0}.
 $$
The left side of this equality is an element of $\K^d$, the $d$ components of which are all 0, and these $d$ relations mean that the polynomial 
$$
b_0+b_1T+\cdots+b_{d-1}T^{d-1}
$$
vanishes at the point $\gamma_j$ with multiplicity at least $t_j$ for $1\le j\le \ell$. Since $t_1+\cdots+t_\ell=d$, we deduce that $b_0=\cdots =b_{d-1}=0$.

The determinant of $A$ was calculated in \cite{Lev2}:
$$
\det A=\prod_{1\le i<j\le\ell} (\gamma_j-\gamma_i)^{t_it_j}.
$$

\sk 

 \indent 
 \subsection{\bf Interpolation.}

The matrix $A$ is associated with the linear system of $d$ equations in $d$ unknowns which amounts to finding a polynomial $f\in \K[z]$ of degree $<d$ for which the $d$ numbers 
$$
 \frac{\rmd^i f }{\rmd z^i} (\gamma_j), \qquad (1\le j\le \ell,\; 0\le i\le t_j-1)
$$
take prescribed values.
Sharp estimates related with this linear system are provided by Lemma 3.1 of \cite{NguyenRoy}.

 Before stating and proving the next proposition, we introduce the following notation. 

Let $g\in \K(z)$, let $z_0\in \K$ and let $t\ge 1$. Assume $z_0$ is not a pole of $g$. We set
$$
T_{g,z_0,t}(z)=\sum_{i=0}^{t-1} 
 \frac{\rmd^i g}{\rmd z^i}(z_0)
\frac{(z-z_0)^i}{i!} \cdotp
$$
In other words,
 $T_{g,z_0,t}$ is the unique polynomial in $\K[z]$ of degree $<t$ such that there exists $r(z)\in \K(z)$ having no pole at $z_0$ with 
$$
g(z)= T_{g,z_0,t}(z)+(z-z_0)^t r(z).
$$ 
Notice that if $g$ is a polynomial of degree $<t$, then $g=T_{g,z_0,t}$ for any $z_0\in \K$.

\begin{proposition}\label{Proposition:Interpolation}
Let $\gamma_j$ ($1\le j\le \ell$) be distinct elements in $\K$, $t_j$ ($1\le j\le \ell$) be positive integers, $\eta_{ij}$ ($1\le j\le \ell$, $0\le i\le t_j-1$) be elements in $\K$. Set $d=t_1+\cdots+t_\ell$. There exists a unique polynomial $f \in \K[z]$ of degree $<d$ satisfying 
\begin{equation}\label{Equation:InterpolationProblem}
 \frac{\rmd^i f}{\rmd z^i} (\gamma_j) =\eta_{ij}, \qquad (1\le j\le \ell,\; 0\le i\le t_j-1).
\end{equation}
For $j=1,\dots,\ell$, define
$$
h_j(z)=\prod_{\atop{1\le k\le \ell}{k\not= j}} 
\left(
\frac{z-\gamma_k}{\gamma_j-\gamma_k}
\right)^{t_k}
\quad\hbox{ 
and 
}\quad 
p_j(z)=\sum_{i=0}^{t_j-1}\eta_{ij} \frac{ (z-\gamma_j)^i}{i!}\cdotp
$$
Then the solution $f$ of the interpolation problem $(\ref{Equation:InterpolationProblem})$ is given by
\begin{equation}\label{Equation:InterpolationFormula}
f=\sum_{j=1}^\ell h_j T_{\frac{p_j}{h_j},\gamma_j,t_j}.
\end{equation}
\end{proposition}

\begin{proof} 
The conditions $(\ref{Equation:InterpolationProblem})$ can be written 
$$
T_{f,\gamma_k,t_k}=p_k \quad\hbox{for}\quad k=1,\dots,\ell.
$$
The unicity is clear: the difference between two solutions is a polynomial of degree $<d$ which vanishes at $d$ points (including multiplicity), hence is the zero polynomial.

Since $h_j(\gamma_j)=1$, the quantity $q_j=T_{ \frac{p_j}{h_j},\gamma_j,t_j }$ is well defined and is a polynomial of degree $<t_j$. 
Since $h_j$ is a polynomial of degree $d-t_j$, the polynomial $f$ in $(\ref{Equation:InterpolationFormula})$, namely
$$
f=h_1q_1+\cdots+h_\ell q_\ell,
$$ 
is a polynomial of degree $< d$. Let us prove  that this polynomial $f$ verifies the equalities in  $(\ref{Equation:InterpolationProblem})$.
For $1\le k\not=j\le \ell$ and $0\le i\le t_k-1$, we have 
$$
 \frac{\rmd^i h_j}{\rmd z^i} (\gamma_k) =
0,
$$
and therefore also
$$
 \frac{\rmd^i (h_j q_j)}{\rmd z^i} 
(\gamma_k)=0.
$$
Hence, for the function $f$ given by $(\ref{Equation:InterpolationFormula})$ and for $1\le k \le \ell$, $0\le i\le t_k-1$, we have 
$$
 \frac{\rmd^i f}{\rmd z^i} (\gamma_k) =
 \frac{\rmd^i (h_k q_k)}{\rmd z^i}(\gamma_k).
$$ 
In other words, for $1\le k \le \ell$, we have 
$$
T_{f,\gamma_k,t_k}=T_{h_kq_k,\gamma_k,t_k}.
$$
By definition of $T$, the function $\displaystyle q_k- \frac{p_k}{h_k}$ has a zero of multiplicity $\ge t_k$ at $\gamma_k$, hence the same is true for the function $h_kq_k-p_k$.
Therefore, for any $k\in \{1,\ldots, \ell\}$, we have 
$$
T_{h_kq_k,\gamma_k,t_k}=p_k,
$$ 
whereupon, $T_{f,\gamma_k , t_k}=p_k$.
This completes the proof.
\end{proof}

The Lagrange--Hermite interpolation formula \cite{Hermite} deals with this question when $\K=\C$ and when the values $\eta_{ij}$ are of the form
$$
\eta_{ij}=
 \frac{\rmd^i F}{\rmd z^i} (\gamma_j) \qquad (1\le j\le \ell,\; 0\le i\le t_j-1)
$$
for a function $F$ which is analytic in a domain containing the points $\gamma_1,\dots,\gamma_\ell$. 

\begin{proposition}\label{Proposition:HermiteFormula}
Let $D$ be a domain in $\C$, $F$ an analytic function in $D$, $\gamma_1,\dots,\gamma_\ell$ distinct points in $D$  
and  $\Gamma$ a simple curve inside which the points $\gamma_1,
\dots , \gamma_\ell$ are located. 
Then the unique polynomial $f \in \C[z]$ of degree $<d$ satisfying 
$$
 \frac{\rmd^i f}{\rmd z^i} (\gamma_j) =
 \frac{\rmd^i F}{\rmd z^i} (\gamma_j)
, \qquad (1\le j\le \ell,\; 0\le i\le t_j-1)
$$
is given, for $z$ inside $\Gamma$, by
$$
f(z)=F(z)+\frac{1}{2i\pi} \int_\Gamma
\Phi(\zeta)
d\zeta
$$ 
with 
$$
\Phi(\zeta)=
\frac{
F(\zeta) }
{z-\zeta}
\prod_{j=1}^\ell \left(\frac{z-\gamma_j}{\zeta-\gamma_j}\right)^{t_j}.
$$
\end{proposition}

\begin{proof}
The residue at $\zeta=z$ of $\Phi(\zeta)$ is $-F(z)$. Under the assumptions of Proposition \ref{Proposition:HermiteFormula} and with the notations of Proposition \ref{Proposition:Interpolation}, we have 
$$
p_j=T_{F,\gamma_j,t_j}.
$$
It remains to show that for $1\le j\le \ell$, the residue at $\zeta=\gamma_j$ of $\Phi(\zeta)$ is 
$$
h_j (z)T_{\frac{p_j}{h_j},\gamma_j,t_j}(z).
$$
We first notice that for $m\in\Z$ and $t\in\Z$ with $t\ge 0$, the residue at $\zeta=0$ of 
$$
\zeta^m \left(\frac{z}{\zeta}\right)^t \frac{1}{z-\zeta}
$$
is $z^m$ for $m\le t-1$ and $z\not=0$, and is $0$ otherwise, namely for $z=0$ as well as for $m\ge t$.
Therefore, when $\varphi(\zeta)$ is analytic at $\zeta=\gamma$, the residue at $\zeta=\gamma$ of 
$$
\varphi(\zeta) \left(\frac{z-\gamma}{\zeta-\gamma}\right)^t \frac{1}{z-\zeta}
$$
is $T_{\varphi,\gamma,t}(z)$. Since 
$$
\Phi(\zeta)=\frac{
F(\zeta) }
{z-\zeta}
 \left(\frac{z-\gamma_j}{\zeta-\gamma_j}\right)^{t_j}
\frac{h_j(z)}{h_j(\zeta)},
$$
and since $h_j(\gamma_j)\not=0$, 
the residue at $\zeta=\gamma_j$ of $\Phi(\zeta)$ is 
$$
h_j (z)T_{\frac{F}{h_j},\gamma_j,t_j}(z).
$$
Finally, we notice that when $\varphi_1$ and $\varphi_2$ are analytic at $\gamma$, then 
$T_{\varphi_1\varphi_2,\gamma,t}=T_{\widetilde{\varphi}_1\varphi_2,\gamma,t}$ with 
$\widetilde{\varphi}_1=T_{\varphi_1,\gamma,t}$. This final remark with $\gamma=\gamma_j$, $t=t_j$, $\varphi_1=F$, $\widetilde{\varphi}_1=p_j$, $\varphi_2=1/h_j$ completes the proof. 

\end{proof}

There are other formulae for the solution to the interpolation problem $(\ref{Equation:InterpolationProblem})$. For instance, writing $t_j$ times each $\gamma_j$, one gets a sequence $z_1,\dots,z_d$, and the so--called
{\it Newton's divided differences interpolation polynomials} give formulae for the coefficients $c_0,\dots,c_{d-1}$ in
$$
f(z)=c_0+c_1(z-z_1)+c_2(z-z_1)(z-z_2)+\cdots+c_{d-1}(z-z_1)(z-z_2)\cdots(z-z_{d-1}).
$$

 \indent
 \subsection{\bf Polynomial combinations of powers.}

From the preceding sections, we deduce that the linear recurrence sequences over an
 algebraically closed field of characteristic 0 are in bijection with the linear combinations 
 of the powers $\gamma_j^a$ ($1\leq j \leq \ell$)
 with polynomial coefficients of the form
\begin{equation}\label{Equation:u(a)}
u(a)=\sum_{j=1}^\ell \sum_{i=0}^{t_j-1} v_{ij} a^i \gamma_j^a \qquad (a\in\Z).
\end{equation}
The piece of data 
 $\bfc=(c_1,\dots,c_d)\in \K^d$ is equivalent to being given $\ell$ distinct nonzero complex numbers 
$\gamma_1,\dots,\gamma_\ell$ and $\ell$ positive integers $t_1,\dots,t_\ell$ together with the property that
$$
T^d-c_1T^{d-1}-\cdots-c_{d-1}T-c_d
= \prod_{j=1}^\ell (T-\gamma_j)^{t_j}
$$
with $d=t_1+\cdots+t_\ell$. \sk

A change of basis for $\K^d$, involving the transition matrix 
$$
\left(
a^i \gamma_j^a
\right)_
{\atop{0\le a\le d-1}{1\le j\le \ell,0\le i\le t_j-1}},
$$
 allows to switch from the initial conditions $u(a)$ for $0\le a\le d-1$ to the $d$ coefficients $v_{ij}$ of $(\ref{Equation:u(a)})$.

Since 
$$
\frac{1}{1-\gamma_j z}=\sum_{a\ge 0} (\gamma_j z)^a
$$
and 
$$
 \left( z\frac{\rmd}{\rmd z}\right)^i 
(\gamma_j z)^a=a^i(\gamma_jz)^a,
$$ 
the generating function of the sequence $\bigl(u(a)\bigr)_{a\in\Z}$ given by $(\ref{Equation:u(a)})$ is
$$
U(z)= \sum_{a\ge 0} u(a) z^a=\sum_{j=1}^\ell \sum_{i=0}^{t_j-1} v_{ij} \left( z\frac{\rmd}{\rmd z}\right)^i 
\left(\frac{1}{1-\gamma_j z}\right),
$$
 which is a rational fraction with denominator $\displaystyle\prod_{j=1}^\ell (1-\gamma_j z)^{t_j}$, as expected.

 \bigskip
 \noindent
 \subsection{\bf The ring of linear recurrence sequences.}

A sum and a product of two polynomial combinations of powers is still a polynomial combination of powers.
If $U_1$ and $U_2$ are two linear recurrence sequences of characteristic polynomials
$P_1$ and $P_2$ respectively, then $U_1+U_2$ satisfies the linear recurrence, the
 characteristic polynomial of which is 
$$
\frac{P_1P_2}{\gcd(P_1,P_2)}\cdotp
$$ 
Consequently, the union of all vector spaces $E_{\bfc}$, with $\bfc$ 
running through the set 
of $d$--tuples $(c_1,\dots,c_d)\in \K^d$ subject to $c_d\neq 0$, and $d$ running through the set of integers $\ge 1$, is still a vector subspace of $\K^\Z$.\sk

Moreover, if the characteristic polynomials of the two linear recurrence sequences $U_1$ and $U_2$ are respectively
$$
P_1(T)=\prod_{j=1}^{\ell}(T-\gamma_j)^{t_j}
\quad
\hbox{and}
\quad
P_2(T)=\prod_{k=1}^{\ell'}(T-\gamma'_k)^{t'_k},
$$
then $U_1U_2$ satisfies the linear recurrence, the characteristic polynomial of which is 
$$
 \prod_{j=1}^{\ell} \prod_{k=1}^{\ell'} (T-\gamma_j\gamma'_k)^{t_j+t'_k-1}.
$$ 
As a consequence, the linear recurrence sequences form a ring.

\

\subsection{ Non homogeneous linear recurrence sequences}\label{SS:SuitesRecurrentesNonHomogenes}

Let us suppose now that a factorisation of the characteristic polynomial $P(T)$ of a linear recurrence relation 
 is of the form $P=QR$, with $R$ completely decomposed in $\K[T]$. Let us write
 $$
 P(T)=T^d-\sum_{i=1}^d c_iT^{d-i},
 \quad
 Q(T)=T^m-\sum_{i=1}^m b_iT^{m-i},
 \quad
R(T)=\prod_{j=1}^\ell (T-\gamma_j)^{t_j}.
$$
 Hence $d=m+t_1+\cdots+t_\ell$. Then the elements of $E_{\bfc}$
 are the sequences $\bigl(u(a))_{a\in\Z}$ for which there exist $d-m$ elements 
 $$
 \hbox{$\lambda_{ij}$ \qquad ($1\le j\le \ell$, $\quad 0\le i\le t_j-1$) }
 $$
 in $\K$ such that 
 \begin{equation}\label{Equation:relationrecurrenceNONhomogene}
 u(a+m)=b_1u(a+m-1)+\cdots+b_mu(a)+
\sum_{j=1}^\ell \sum_{i=0}^{t_j-1} \lambda_{ij}a^i \gamma_j^a.
\end{equation} 
In order to define an element $\bigl(u(a)\bigl)_{a\in\Z}$ of $E_{\bfc}$ by using the homogenous recurrence relation in $(\ref{Equation:relationrecurrencehomogene})$, we have to give $d$ initial values, for instance $u(0),\dots,{u(d-1)}$. In order to define this sequence by using the non homogeneous recurrence relation $(\ref{Equation:relationrecurrenceNONhomogene})$, it is sufficient to have $m$ initial conditions, say $u(0),\dots,u(m-1)$, but we also have to know the elements $\lambda_{ij}$ for $1\le j\le \ell$ and $0\le i\le t_j-1$ (which altogether are $d$ conditions, as is required in a vector space of dimension $d$).\sk

Consider the transition matrix associated to the change of basis, allowing to switch from the initial conditions 
$$
\hbox{$u(a)\;$ for $\;0\le a\le d-1$}
$$
 to the initial conditions 
 $$
 \hbox{$u(a)\;$ for $\;0\le a\le m-1\;$ and $\;\lambda_{ij}\;$ for $\; 1\le j\le \ell\;$ and $\;0\le i\le t_j-1$}.
 $$
 It is a matrix which has only a diagonal of two blocks, 
$$
\left(
\begin{matrix}
I_m&0\\
0&A
\end{matrix}
\right)
\quad 
\hbox{ with }\quad
A=\left(
\begin{matrix}
A_1
\\
\vdots
\\
A_\ell
\end{matrix}
\right).
$$
The first block $I_m$ is the $m\times m$ identity matrix. The second block $A$ is a generalized Vandermonde matrix similar to the matrix in $(\ref{Equation:matriceA})$ made of the blocks $A_1,\dots, A_\ell$ described in $(\ref{blocks})$.\sk
 
A particular case is the trivial one when 
 $P=Q$, $m=d$ and $R=1$. Another one is when 
 $P=R$, $Q=1$ and $m=0$, which corresponds to the case studied in Section $\ref{SS:PolynomeCaracteristiqueDecompose}$.
 \sk
 
 \indent
 {\bf Example.} Let us consider
 $$P(T)=(T-\gamma)^2, \quad Q(T)=R(T)=T-\gamma.$$
 There are three ways of defining an element $\big( u(a)\bigr)_{a\in\Z}$ of the vector space $E_{\bfc}$ when $\bfc=(2,-1)$. 
 The first one is to mention that the sequence satisfies the binary linear recurrence relation 
 $$
 u(a+2)=2 u(a+1) - u(a) \quad (a\in\Z)
 $$
and give two initial values, for, say $u(0)$ and $u(1)$. The second one is to write 
 $$
 u(a)=(\lambda_1+\lambda_2 a) \gamma^a \quad (a\in\Z)
 $$
 and give the values of $\lambda_1$ and $\lambda_2$. The third one is in-between the previous ones; one writes that the sequence satisfies 
 $$
 u(a+1)=\gamma u(a)+\lambda \gamma^a \quad (a\in\Z)
 $$ 
while providing an initial value, for, say $u(0)$, and the value of $\lambda$. 

\subsection{ Exponential polynomials}\label{SS:PolynomesExponentiels}

The sequence of derivatives of an exponential polynomial evaluated at one point satisfies a linear recurrence relation. This allows us to deduce the following well known result (Ch. I, \S 7 of \cite{Siegel}).

\begin{lemma}\label{Lemma:ExponentialPolynomials}
Let $a_1(z),\dots,a_\ell(z)$ be nonzero polynomials of $\C[z]$ of degrees smaller than $ t_1,\dots,t_\ell$ respectively. Let $\gamma_1,\dots,\gamma_\ell$ be distinct complex numbers. Let us suppose that the function 
$$
F(z)=a_1(z)e^{\gamma_1 z}+\cdots+ a_\ell(z)e^{\gamma_\ell z}
$$
is not identically $0$. Then its vanishing order at a point $z_0$ is smaller than or equal to $t_1+\cdots+t_\ell -1$.
\end{lemma}

\begin{proof}
Define $d=t_1+\cdots+t_\ell$. We give two proofs of Lemma $\ref{Lemma:ExponentialPolynomials}$. A short one by induction on $d$ is as follows. For $d=1$ we have $\ell=1$ and $F$ has no zero. Assume $\ell\ge 2$. Without loss of generality we may assume $\gamma_1=0$. If $F$ has a zero of multiplicity $\ge T_0$ at $z_0$, then $F(z)-a_1(z)$ has a zero of multiplicity $\ge T_0-t_1$ at $z_0$. The result follows. 

\par
Our second proof relates Lemma $\ref{Lemma:ExponentialPolynomials}$ with linear recurrence sequences. We now assume $\gamma_1,\dots,\gamma_\ell$ all nonzero, as we may without loss of generality. 
Write the Taylor expansion of $F(z+z_0)$ at $z=0$:
$$
F(z+z_0)=\sum_{a\ge 0} \frac{u(a)}{a!} z^a.
$$
Let us show that the sequence $(u(0),u(1),\dots,u(a),\dots)$ satisfies a linear recurrence relation of order $\le d$. Define $a_{ij}\in \C$ by
$$
a_j(z+z_0)e^{\gamma_j z_0}=\sum_{i=0}^{t_j-1} a_{ij} z^i \quad (1\le j\le \ell),
$$
so that
$$
F(z+z_0)=\sum_{j=1}^\ell 
\sum_{i=0}^{t_j-1} a_{ij} z^i e^{\gamma_j z}.
$$
Since $\gamma_j\not=0$ for $j=1,\dots,\ell$, 
$$
u(a)=\sum_{j=1}^\ell \sum_{i=0}^{t_j-1} a_{ij} a(a-1)\cdots(a-i+1) \gamma_j^{a-i}
$$
has the same form as in $(\ref{Equation:u(a)})$. Therefore the sequence $\big( u(a)\bigr)_{a\in\Z}$ satisfies a linear recurrence relation of order $\le d$. 
It follows that the conditions 
$$u(0)=\cdots=u(d-1)=0$$
 imply $u(a)=0$ for any $a\ge 0$. 
\end{proof} 

We can state this lemma in the following way: When the complex numbers $\gamma_j$ are distinct, the determinant 
$$
\left|
\left(\frac{\rmd}{\rmd z}\right)^a \bigl(z^i e^{\gamma_jz}\bigr)_{z=0}
\right|
_{\atop{0\le i\le t_j-1, \; 1\le j\le \ell}{0\le a\le d-1}}
$$
is different from $0$. This is no surprise that we come across the determinant of the matrix $(\ref{Equation:matriceA})$.

\section{Families of binary forms }\label{S:FamillesFormesBinaires}

The equations $(\ref{Equation:FaSomme})$ and $(\ref{Equation:FaProduit})$ give, for $1\le h\le d$ and $a\in\Z$,
\begin{equation}\label{EquationUh(a)}
U_h(a)=\sum_{1\le i_1<\cdots<i_h\le d} \alpha_{i_1}\cdots \alpha_{i_h} (\epsilon_{i_1}\cdots \epsilon_{i_h})^a.
\end{equation}
For example, for $a\in\Z$,
$$
U_1(a)=\sum_{i=1}^d \alpha_i\epsilon_i^a,
\quad
U_d(a)=\prod_{i=1}^d \alpha_i\epsilon_i^a.
$$

The relations $(\ref{EquationUh(a)})$ show that for $1\le h\le d$, the sequence $\bigl(U_h(a)\bigr)_{a\in\Z}$ is a linear combination 
of the sequences 
$$\bigl( (\epsilon_{i_1}\cdots \epsilon_{i_h})^a\bigr)_{a\in\Z}, \quad (1\le i_1<\cdots<i_h\le d). 
$$ 
For $1\le h\le d$, consider the set 
$$
\calE_h=
\{
\epsilon_{i_1}\cdots \epsilon_{i_h}\; \mid\; 1\le i_1<\cdots<i_h\le d
\}
$$
and note $m_h$ its cardinality. The elements of $\calE_h$ are values of monomials in $m_1$ variables of degree $h$. 
The map from $\calE_h$ to $\calE_{d-h}$ defined by
$$\eta\mapsto \epsilon_1\cdots\epsilon_d\eta^{-1}$$
is a bijection and we have 
$$
m_h=m_{d-h}\le\min\left\{ \binom{d}{h},
\;
 \binom{m_1+h-1}{h},
 \; \binom{m_1+d-h-1}{d-h}\right\}.
$$

The sequence $\bigl(U_h(a)\bigr)_{a\in\Z}$ satisfies the linear recurrence relation of order $m_h$ with the characteristic polynomial
$$
\prod_{\eta\in\calE_h}( T-\eta).
$$
This polynomial is also written as 
$$
\prod_{\eta\in\calE_{d-h}} (T- \epsilon_1\cdots\epsilon_d\eta^{-1}),
$$
which is matching $(\ref{EquationUh(a)})$ via
$$ 
U_h(a)=U_d(a) \sum_{1\le j_1<\cdots<j_{d-h}\le d} (\alpha_{j_1}\cdots \alpha_{j_{d-h}})^{-1} (\epsilon_{j_1}\cdots \epsilon_{j_{d-h}})^{-a}.
$$
For example, the sequence $\bigl(U_{d-1}(a)\bigr)_{a\in\Z}$ satisfies the linear recurrence relation of order $d$, the characteristic polynomial of which is 
$$
\prod_{i=1}^d (T-\epsilon_1\cdots \epsilon_d \epsilon_i^{-1})
=
(T-\epsilon_2\cdots \epsilon_d)(T-\epsilon_1\epsilon_3\cdots \epsilon_d)\cdots(T-\epsilon_1\cdots \epsilon_{d-1}).
$$

The case $\epsilon_1=\ldots=\epsilon_d$ is trivial: we have 
$$
U_h(a)=\epsilon_1^a U_h(0)=(-1)^h a_h \epsilon_1^a,
$$
 and each of the sequences $\bigl(U_h(a)\bigr)_{a\in\Z}$ satisfies
$$
U_h(a+1)=\epsilon_1U_h(a).
$$

Let us consider the example 
$$
\epsilon_1=\ldots=\epsilon_\ell=\epsilon, \quad 
\epsilon_{\ell+1}=\ldots=\epsilon_d=\eta,
$$
with $\epsilon$ and $\eta$ being two distinct complex numbers. We have 
$$
\calE_1=\{\epsilon,\eta\},
\qquad
\calE_{d-1}=\{\epsilon^{\ell-1} \eta^{d-\ell}, \epsilon^{\ell} \eta^{d-\ell-1}\}
$$
and
$$
\calE_2=\{\epsilon^2,\epsilon\eta,\eta^2\},
\quad
\calE_{d-2}=\{\epsilon^{\ell-2} \eta^{d-\ell}, \epsilon^{\ell-1} \eta^{d-\ell-1}, \epsilon^{\ell} \eta^{d-\ell-2}\}.
$$
The sequence $\bigl(U_1(a)\bigr)_{a\in\Z}$ satisfies the binary recurrence relation, the characteristic polynomial of which is
$$
(T-\epsilon)(T-\eta);
$$
the sequence $\bigl(U_{d-1}(a)\bigr)_{a\in\Z}$ satisfies the binary recurrence relation,
 the characteristic polynomial of which is
$$
(T-\epsilon^{\ell-1} \eta^{d-\ell})(T-\epsilon^{\ell} \eta^{d-\ell-1}),
$$
while the sequence $\bigl(U_2(a)\bigr)_{a\in\Z}$ satisfies the ternary recurrence relation,
 the characteristic polynomial of which is 
$$
(T-\epsilon^2)(T-\eta^2)(T-\epsilon\eta).
$$
In particular, if one writes 
$$
(T-\epsilon^2)(T-\eta^2)=T^2-AT-B,
$$
then there exists a constant $C\in\C$ such that, for any $a\in\Z$, one has 
$$
U_2(a+2)=AU_2(a+1)+BU_2(a)+C (\epsilon\eta)^a.
$$
Finally, the sequence $\bigl(U_{d-2}(a)\bigr)_{a\in\Z}$ satisfies the ternary recurrence relation, the characteristic polynomial of which is 
$$
(T-\epsilon^{\ell-2}\eta^{d-\ell})
(T-\epsilon^{\ell-1}\eta^{d-\ell-1})
(T-\epsilon^{\ell}\eta^{d-\ell-2}).
$$

\section*{Acknowledgments} 
{\small 
The authors want to thank the organizers of the {\sl Second International Conference on Pure and Applied Mathematics} held in Goroka (Papua New Guinea) and to express their gratitude to   
Samuel Kopamu for having invited them to give lectures. We are also thankful to Pietro Corvaja for two courses he gave on linear recurrence sequences, one in Bamako (Mali) in 2010 and the next one in Kozhikode (Kerala, India) in 2013. 
 
}

\vfill

 \vfill
 
\hbox{
\small
\vbox{
\hbox{Claude LEVESQUE}
	\hbox{D\'{e}partement de math\'{e}matiques et de statistique
	}
	\hbox{Universit\'{e} Laval
	}
	\hbox{Qu\'{e}bec (Qu\'{e}bec)
	}
	\hbox{CANADA G1V 0A6
	}
	\hbox{Claude.Levesque@mat.ulaval.ca
	}
}	
\hfill
\vbox{	\hbox{Michel WALDSCHMIDT 
	}
	\hbox{Sorbonne Universit\'es
	}
	\hbox{UPMC Univ Paris 06
	}
	\hbox{UMR 7586 IMJ-PRG
	}
	\hbox{F -- 75005 Paris, France 
	}
	\hbox{michel.waldschmidt@imj-prg.fr}
}	
}	

\vfill

\end{document}